\documentclass[12pt]{amsart}

\usepackage{amsmath, amstext, amsthm, amssymb, amsfonts, latexsym}

\numberwithin{equation}{section}

\newcommand{\T}{\mathbb{T}}
\newcommand{\R}{\mathbb{R}}
\newcommand{\Z}{\mathbb{Z}}
\newcommand{\SSS}{\mathbb{S}}

\newcommand{\RR}{\mathbb{R}^2}

\newcommand{\TTT}{\mathbb{T}^3}
\newcommand{\RRR}{\mathbb{R}^3}
\newcommand{\ZZZ}{\mathbb{Z}^3}

\newcommand{\tx}{\tilde{x}}
\newcommand{\ty}{\tilde{y}}
\newcommand{\tz}{\tilde{z}}
\newcommand{\tp}{\tilde{p}}
\newcommand{\tq}{\tilde{q}}
\newcommand{\tgamma}{\tilde{\gamma}}

\newcommand{\tGamma}{\tilde{\Gamma}}
\newcommand{\tSigma}{\tilde{\Sigma}}
\newcommand{\tD}{\tilde{D}}
\newcommand{\tW}{\tilde{W}}
\newcommand{\tf}{\tilde{f}}

\newcommand{\cF}{\mathcal{F}}
\newcommand{\tcF}{\tilde{\mathcal{F}}}

\newcommand{\GL}{\operatorname{GL}}

\newtheorem{theorem}{Theorem}[section]

\newtheorem{lemma}[theorem]{Lemma}

\title{Transitivity of conservative diffeomorphisms isotopic to Anosov on $\TTT$}
\date{\today}
\author[Martin Andersson]{Martin Andersson}
\email{martin@mat.uff.br}

\author[Shaobo Gan]{Shaobo Gan}
\thanks{Shaobo Gan  is supported by 973 project (Grant No. 2011CB808002) and NSFC (Grant No. 11025101 and 11231001) }
\email{gansb@pku.edu.cn}

\begin{document}

\begin{abstract}
We prove transitivity for  volume preserving $C^{1+}$ diffeomorphisms on $\TTT$ which are isotopic to a linear Anosov automorphism along a path of weakly partially hyperbolic diffeomorphisms.
\end{abstract}

\maketitle

\section{introduction}

Ergodicity of volume preserving diffeomorphisms with some hyperbolicity has been one of the main topics of research in differentiable dynamics during the last two decades. Nearly all results that have been found are in the setting of strong partial hyperbolicity, i.e., partial hyperbolicity informally referred to as type $E^s \oplus E^c \oplus E^u$. A notable exception to this rule is Tahzibi's example \cite{MR2085722} of stably ergodic diffeomorphisms on $\T^4$ admitting only a dominated splitting. Another exception is \cite[Corollary 1.8]{MR2574879}, where it was proved that ergodicity is an open phenomenon among volume preserving partially hyperbolic diffeomorphisms of type $E^{cs} \oplus E^u$ if $E^{cs}$ is mostly contracting (see \cite{BV} for definitions). Unfortunately, the  mostly contracting condition seems too restrictive and is too hard to work with in any case, so one cannot expect it to provide a general route toward stable ergodicity.

If one wishes to extend the search of stable ergodicity outside the classical realm of strong partial hyperbolicity, one natural down to earth question to ask is whether every sufficiently smooth volume preserving diffeomorphism with dominated splitting on $\TTT$, isotopic to a linear Anosov diffeomorphism, is ergodic. There is no framework in today's mathematics in which this question can be approached in its full generality. Here we prove that such diffeomorphisms are transitive, under a mild condition on the isotopy between the diffeomorphism and its linear representative.

For volume preserving diffeomorphisms in three dimensions, dominated splitting is equivalent to  weak partial hyperbolicity, i.e., hyperbolicity informally referred to as type $E^s \oplus E^{cu}$ or $E^{cs} \oplus E^u$, with $E^s$ and $E^u$ being one dimensional in each case. Our condition on the isotopy is that if the linear automorphism $A$ to which our diffeomorphism $f$ is isotopic has one dimensional stable direction, then $f$ is of type $E^s \oplus E^{cu}$ and, moreover, the isotopy can be chosen so that each diffeomorphism along its path is also of type $E^s \oplus E^{cu}$. Similarly, if $A$ has a one dimensional unstable direction, then the isotopy can be chosen so that each diffeomorphism along its path (including $f$ itself) is of type $E^{cs} \oplus E^u$. We do not require that the isotopy is contained in the set of volume preserving diffeomorphisms.

In the case where the diffeomorphism $f$ is of class $C^2$ and strongly partially hyperbolic (type $E^s \oplus E^c \oplus E^u$), transitivity is a corollary of a much stronger theorem by Hammarlindl and Ures \cite{doi:10.1142/S0219199713500387} which states that either $f$ is ergodic or it is topologically conjugate to $A$. The key feature of our approach is, therefore, that it works in the considerably more general setting of weak partial hyperbolicity.

Our proof uses absolute continuity of the strong foliation and therefore requires the diffeomorphism to be of class $C^r$ for some (possibly non-integer) $r>1$. This property is only used in Lemma \ref{nontrivial loop}, and we believe that it is possible to replace the proof of that lemma with a more subtle argument that does not use absolute continuity. If so, the result would hold for $C^1$ diffeomorphisms. However, if one sees our result as an invitation to study stable ergodicity of weakly partially hyperbolic systems, then such an improvement would make little difference.

This work is based on the ideas developed in \cite{1501.01670}. The authors would like to thank Radu Saghin for useful discussions and and Rafael Potrie for advice on weak partial hyperbolicity on $\TTT$.

\section{The result}

Recall that a diffeomorphism $f: \TTT \to \TTT$ has a dominated splitting $F\prec G$ if $F$ and $G$ are $Df$-invariant continuous sections of $T\TTT$ of complementary dimension such that  $T\TTT$ is a direct sum $F \oplus G$, and there exists some $n_0 \geq 1$ such that
\begin{equation}
\|Df^{n_0}_{ \vert F_x} \| \| (Df^{n_0}_{ \vert G_x} )^{-1} \| < \frac{1}{2}
\end{equation}
for every $x \in \TTT$.

We say that $f: \TTT \to \TTT$ is weakly partially hyperbolic with one dimensional strong stable bundle if it has dominated splitting $F\prec G$ such that $\dim F = 1$ and
\begin{equation}
\|Df^{n_0}_{ \vert F_x} \| < \frac{1}{2}
\end{equation}
for some $n_0\geq 1$ and every $x \in \TTT$.

\begin{theorem}\label{main}
Let $A :\TTT \to \TTT$ be a linear Anosov automorphism with one dimensional stable direction. Suppose that $f: \TTT \to \TTT$ is a volume preserving $C^r$ diffeomorphism, with $r>1$, isotopic to $A$ along a path on which each diffeomorphism is weakly partially hyperbolic with one dimensional strong stable bundle. Then $f$ is transitive.
\end{theorem}

It is implicit in the statement that $f$ itself is weakly partially hyperbolic with one dimensional strong stable direction.

Obviously, Theorem \ref{main} can be formulated analogously for the case where both the linear Anosov automorphism $A$ and the weakly partially hyperbolic diffeomorphism $f$ have a one dimensional strong unstable bundle. In this case, the proof follows by replacing $f$ with $f^{-1}$.

\section{Some preliminaries}\label{preliminaries}

We state some preliminary results used in the proof of Theorem \ref{main}.

\begin{theorem}[\cite{HPS, BP}]
Let $f:\TTT \to \TTT$ be a weakly partially hyperbolic $C^r$ diffeomorphism with one dimensional strong stable bundle $E^s$. Then there exists a $C^0$ $f$-invariant foliation $\cF^s$ of $\TTT$ with $C^r$ leaves tangent to $E^s$. Moreover, if $r>1$, then $\cF^s$ is absolutely continuous, and the holonomy between two disks uniformly transverse to $\cF^s$ has bounded Jacobian.
\end{theorem}

See \cite{barreira2013introduction} for a complete and accessible exposition on stable manifolds and absolute continuity.

The foliation $\cF^s$ is called the strong stable foliation of $f$ and its leaves are called strong stable manifolds. The strong stable manifold that contains $x$ will be denoted by $W^s(x)$. For $\delta>0$, we denote by $W_\delta^s(x)$ the set of points in $W^s(x)$ whose distance from $x$ inside the leaf $W^s(x)$ is smaller than $\delta$.

We denote by $\tcF^s$ the lift of $\cF^s$ to the universal cover $\RRR$. If $\tf$ is a lift of $f$, then $\tcF^s$ is $\tf$-invariant. Its leaves are called strong stable manifolds for $\tf$. The strong stable leaf that contains $\tx \in \RRR$ is denoted by $\tW^s(\tx)$.

A matrix $A \in \GL(3,\Z)$ with determinant $\pm 1$ commutes with the canonical covering map $\pi: \RRR \to \TTT$ and therefore induces an automorphism on $\TTT$. If we identify $\pi_1(\TTT)$ with $\ZZZ$ in the obvious way, the action $f_\star$ of $f$ in $\pi_1(\TTT)$ can be represented by an element of $\GL(3,\Z)$. In our context, the matrix of $f_\star$ is equal to $A$ if and only if $f$ is isotopic to the automorphism induced by $A$. Since there is no source of confusion, we use the symbol $A$ to denote i) a matrix, ii) a linear map on $\RRR$, iii) an automorphism on $\pi_1(\TTT)$, and iv) an automorphism of $\TTT$.

\begin{theorem}[Walters \cite{Walters197071} and Franks \cite{MR0271990}]
Let $f$ be a diffeomorphism on $\TTT$ isotopic to a linear Anosov automorphism $A$, and let $\tf:\RRR \to \RRR$ be a lift of $f$ to the universal cover $\RRR$. Then there is a continuous surjection $H:\RRR \to \RRR$, commuting with translations of elements in $\ZZZ$,  satisfying $H\circ\tf =A \circ H $. More specifically, for any $\tx \in \RRR$, $H(\tx)$ is the unique point $\ty$ such that $\|\tf^n(\tx) - A^n(\ty)\|$  is bounded as $n$ ranges over $\Z$.
\end{theorem}

Translations on $\RRR$ by elements of $\ZZZ$ are the deck transformations of the covering map $\pi: \RRR \to \ZZZ$.
Since $H$ is continuous and commutes with deck transformations, it projects to a continuous map $h:\TTT \to \TTT$ homotopy to the identity.
In particular, $H$ is of bounded $C^0$ distance from the identity on $\RRR$.

\begin{theorem}[Potrie \cite{1206.2860}]\label{potries theorem}
Let $f: \TTT \to \TTT$ be a $C^1$ diffeomorphism, isotopic to a linear Anosov automorphism along a path of weakly partially hyperbolic  diffeomorphisms with one dimensional strong stable bundle. Then there exists a foliation $\tcF^{cu}$ tangent to the center-unstable bundle of $f$. Moreover, if $\tcF^{cu}$ is the lift of $\cF^{cu}$ then every leaf of $\tcF^s$ intersects every leaf of $\tcF^{cu}$ in exactly one point.
\end{theorem}

The second part of the statement of Potries theorem is referred to by saying that $\tcF^s$ and $\tcF^{cu}$ have global product structure. One of the consequences of the global product structure is that $\tcF^s$ is quasi-isometric: there exist constants $C, D$ such that if $\tx$ and $\ty$ belong to the same leaf of $\tcF^s$, then the distance between $\tx$ and $\ty$ inside that leaf is smaller than $C \|\tx - \ty \| + D$ \cite[Proposition 6.8]{1206.2860}. Since the distance between $\tx$ and $\ty$ inside a stable leaf grows indefinitely under backwards iteration, it follows that the restriction of $H$ to leaves of $\tcF^s$ homeomorphically onto stable manifolds of $A : \RRR \to \RRR$.

Our proof of Theorem \ref{main} makes use of the following intersection property. It is somewhat analogous to Lemma 4.2 in \cite{1501.01670}.

\begin{lemma}\label{intersection}
Let $X,Y,Z$ be a basis of $\RRR$. Let $\xi: \R \to \RRR$ and $\psi: \RR \to \RRR$ be continuous functions. Suppose there exists $K>0$ such that $\| rX - \xi(r) \|< K$ for every $r \in \R$, and that $\|sY + tZ - \psi(s,t)\| < K$ for every $(s,t) \in \RR$. Then there exist $r,s,t$ such that $\xi(r) = \psi(s,t)$.
\end{lemma}

Although it may seem intuitively obvious, giving rigorous proof of this kind of lemma can be puzzling unless one knows some trick. Here, the trick is to consider the homotopy class of a map from $\SSS^2$ to itself.

\begin{proof}[Proof of Lemma \ref{intersection}]
First notice that, up to a change of basis, we may assume for simplicity that $X,Y,Z$ is the standard basis in $\RRR$, so that $\| \xi(r) - (r,0,0)\| < K$ and $\|\psi(s,t)- (0,s,t)\| <K$. Consider the map
\begin{align}
\Theta_R: \SSS^2 & \to \SSS^2 \\
\Theta_R(u,v,w)  & =  \frac{\psi(Rv,Rw)- \xi(Ru)}{\|\psi(Rv,Rw)- \xi(Ru)\|}.
\end{align}
It is clearly well defined for evey $R$ sufficiently large. Moreover, $\Theta_R$ converges uniformly to the involution $(u,v,w) \mapsto (-u,v,w)$ when $R$ tends to infinity. If Lemma \ref{intersection} would be false, so that we would have $\psi(s,t) \neq \xi(r)$ for every triple $(r,s,t) \in \RRR$, then $\Theta_R$ would be  well defined for every $R \geq 0$. Moreover, $R\mapsto \Theta_R$ would be continuous and $\Theta_0$ would be constant. But the involution $(u,v,w) \mapsto (-u,v,w)$ is not homotopic to a constant, so $\psi(s,t) \neq \xi(r)$ cannot hold for every $(r,s,t) \in \RRR$.
\end{proof}

\section{Invariant regular open sets}

Recall that an open set $U \subset \TTT$ is called \emph{regular} if it is equal to the interior of its closure. If $A$ is any subset of $\TTT$, then we denote by $A^\perp$ the complement of its closure, i.e., $A^\perp = \TTT \setminus \overline{A}$. With this notation, an open set $U \in \TTT$ is regular if and only if $U = U^{\perp \perp}$. It is easily shown that if $U$ is open, then $U \subset U^{\perp \perp}$ and $U^{\perp} = U^{\perp \perp \perp}$. In particular, $U^\perp$ is regular (see \cite[chapter 10]{MR2466574}). Note also that $f(U^\perp) = (f(U))^\perp$ for every homeomorphism $f: \TTT \to \TTT$.

\begin{lemma}\label{complementary pair}
Let $f: \TTT \to \TTT$ be a volume preserving homeomorphism. If $f$ is not transitive, then there exist non-empty $f$-invariant regular open sets $U, V$ with $U \cap V = \emptyset$.
\end{lemma}

\begin{proof}
If $f$ is not transitive, then there exist nonempty open sets $U_0, V_0 \subset \TTT$ such that $U_0 \cap f^n(V_0) = \emptyset $ for every $n \in \Z$. Let $V_1 = \bigcup_{n \in \Z} f^n(V_0)$. Then $V_1$ is open and $f$-invariant. Let $U = V_1^\perp$ and $V = U^\perp = V_1^{\perp \perp}$. Then $U$ and $V$ are non-empty disjoint regular open $f$-invariant sets.
\end{proof}

\begin{lemma}\label{saturated}
Let $f:M \to M$ be a volume preserving weakly partially hyperbolic diffeomorphism with a strong stable bundle. If $U \subset M$ is a regular open $f$-invariant set, then $U$ is saturated by strong stable leaves.
\end{lemma}

\begin{proof}
Since $U = U^{\perp \perp}$, it suffices to prove that, given any open $f$-invariant set $A$, the set $A^\perp$ is saturated by strong stable leaves. But to prove that $A^\perp$ is saturated by strong stable leaves, it is enough to prove that $\overline{A}$ is saturated by strong stable leaves, since $A^\perp$ is the complement of $\overline{A}$ and any subset of $\TTT$ is saturated by strong stable leaves if and only if its complement is saturated by strong stable leaves.

Thus let $A \subset \TTT$ be any open $f$-invariant set, and let  $x \in A$ . Since $f$ is conservative, $x$ is non-wandering.
Therefore, there exists a sequence $x_k$ in $A$ converging to $x$, and $n_k \to \infty$ such that $f^{-n_k}(x_k)$ converges to $x$.
Because $x_k \to x$, there is some $\epsilon>0$ such that $W_\epsilon^s(x_k) \subset A$ for every $k$.
By invariance of $A$, $f^{-n_k}(W_\epsilon^s(x_k))$ is contained in $A$ for every $k$.
Note that the length of $f^{-n_k}(W_\epsilon^s(x_k))$ tends to infinity.
In particular, $f^{-n_k}(W_\epsilon^s(x_k))$  accumulates on $W_L^s(x)$ for any $L>0$.
Therefore, $W_L^s(x) \subset \overline{A}$.
Since $L>0$ is arbitrary, we conclude that $W^s(x) \subset \overline{A}$.
\end{proof}

\begin{lemma}\label{nontrivial loop}
Let $r>1$ and let $f:\TTT \to \TTT$ be a $C^r$ volume preserving weakly partially hyperbolic diffeomorphism with a one dimensional strong stable bundle. Suppose that $U\subset \TTT$ is an $f$-invariant regular open set and denote by $i:U \to \TTT$ the inclusion map. Then $i_\star \pi_1(U)$ is a non-trivial subgroup of $\pi_1(\TTT)$.
\end{lemma}

\begin{proof}
We cover $\TTT$ with a finite number of  sets $C_1, \ldots, C_N$ of the form
\begin{equation}
C_i = \bigcup_{x \in \Sigma_i} W_\delta^s(x).
\end{equation}
Here $\Sigma_i$ are disks uniformly transverse to the the strong stable foliation. By making sure that the $\Sigma_i$ and $\delta$ are not too large, we may (and do) assume that each $C_i$ is simply connected.

Let $D \subset U$ be a disk tangent to the center-unstable bundle and denote by $D_n$ the $n^{th}$ iterate of $D$ under $f$.
By invariance of $U$, $D_n \subset U$ for every $n$.
Denote by $m_n$ the area measure on $D_n$.
Since $f$ is volume hyperbolic, we have $m_n(D_n) \to \infty$ as $n \to \infty$. Now,
\begin{equation}
m_n(D_n) \leq \sum_{i=1}^{N} m_n(D_n \cap C_i),
\end{equation}
so there is at least one $i \in \{1, \ldots, N \}$ such that
\begin{equation}
\limsup_{n \to \infty} m_n(D_n \cap C_i) = \infty.
\end{equation}
 We fix such an $i$ and write $C_i = C$ and $\Sigma_i = \Sigma$.

{\bf Claim:} There is some $n$ and some $x \in \Sigma$ such that $W_\delta^s(x)$ intersects $D_n$ in more than one point.

Consider the holonomy map $h^s: C \to \Sigma$, where $h^s(p)$ is the unique point $q \in \Sigma$ such that $p \in W_\delta^s(q)$. The above claim is equivalent to say that for some $n$, the map $h^s_{\vert C \cap D_n}$ is not injective.

That $h^s_{\vert C \cap D_n}$ cannot be injective for every $n$ is a straightforward consequence of the absolute continuity of the strong stable foliation. Indeed, if $h^s_{\vert C\cap D_n}$ is injective, then $m_\Sigma (h^s (C\cap D_n)) \geq K m_n(C \cap D_n)$ for some constant $K>0$ which does not depend on $n$. But $m_{\Sigma}$ is a finite measure, while $\limsup_{n \to \infty} m_n(C\cap D_n) = \infty$, a contradiction. This proves the claim.

Let $p$ and $q$ be distinct points on $D_n$ that lie in the same stable leaf. Then there is a path $\gamma_1$ from $p$ to $q$ inside $D_n$ and a path $\gamma_2$ from $q$ to $p$ inside $W^s(q)$. Let $\gamma = \gamma_1 * \gamma_2$. Then $\gamma$ is a loop based on $p$. We claim that $\gamma$ is not homotopic to the constant path at $p$. To this end, choose $\tp_1$ in the fiber of $p$ and let $\tgamma$ be the lift of $\gamma$ such that $\tgamma(0) = \tp_1$. Let $\tp_2 = \tgamma(1)$. Since $\gamma$ is a loop, $\tp_2$ is also in the fiber of $p$. To say that $\gamma$ is not homotopic to a constant path is equivalent to say that $\tp_1 \neq \tp_2$.

To see why $\tp_1 \neq \tp_2$, let $\tD_n$ a lift of $D_n$ that contains $\tp$. Let $\tgamma_1$ be the lift of $\gamma_1$ starting at $\tp_1$. Then $\tgamma_1$ ends at a point $\tq$ which is in the fiber of $q$. Moreover, $\tgamma_1$ is contained in $\tD_n$. Let $\tgamma_2$ be a lift of $\gamma_2$ starting at $\tq$. Then $\tgamma_2$ terminates at $\tp_2$. Recall that $\tcF^s$ and $\tcF^{cu}$ have global product structure. Since $\tgamma_1$ lies inside a leaf of $\tcF^{cu}$ and $\tgamma_2$ lies inside a leaf of $\tcF^s$, the image of $\tgamma_1$ and the image of $\tgamma_2$ can have at most one point in common. Since $\tq$ is such a point, we conclude that $\tp_1$ and $\tp_2$ must be distinct.

\end{proof}

\section{Proof of theorem \ref{main}}
Let $f$ be as in Theorem \ref{main} and suppose, for the purpose of obtaining a contradiction,  that $f$ is not transitive. By Proposition \ref{complementary pair} there exist non-empty disjoint $f$-invariant regular open sets $U, V \subset \TTT$. According to Lemma \ref{saturated},  these are saturated by strong stable leaves. Moreover, by Lemma \ref{nontrivial loop}, there is a loop $\gamma$ in $U$ and a loop $\sigma$ in $V$, neither of which is homotopic to a constant path. Let $[\gamma]$ and $[\sigma]$ be the elements of $\pi_1(\TTT)$ represented by $\gamma$ and $\sigma$, respectively. Note that $[\gamma]$ and $[\sigma]$ cannot be eigenvectors of $A = f_\star$, because we are assuming that $A$ is hyperbolic, and a hyperbolic linear map on $\RRR$ that preserves $\ZZZ$ cannot have any eigenvectors in $\ZZZ$. Of course, by invariance of $U$, $f(\gamma)$ is also a loop in $U$ and $[f(\gamma)] = A [\gamma]$. Thus, up to replacing $\gamma$ with $f(\gamma)$ if necessary, we may assume that $[\gamma]$ and $[\sigma]$ are linearly independent when seen as elements of $\RRR$.
Let  $v \in \RRR$ be a stable unit eigenvector of $A$. 
\begin{lemma}
$[\gamma]$, $[\sigma]$ and $v$ are linearly independent.
\end{lemma}

\begin{proof}
Since $[\gamma], [\sigma]\in\ZZZ$ are linearly independent, they span a plane in $\RRR$ which projects to a torus on $\TTT$. In particular, it is not dense. On the other hand, we know that the line spanned by $v$ projects to a dense subset of $\TTT$. Therefore, it cannot be contained in the torus spanned by $[\gamma]$ and $[\sigma]$.
\end{proof}

Let  $\Gamma, \Sigma: \R \to \TTT$ be periodic extensions of $\gamma$ and $\sigma$; that is, $\Gamma(t+n) = \gamma(t)$ and $\Sigma(t+n) = \sigma(t)$ for every $t \in [0,1]$ and every $n \in \Z$. Let $\tGamma$ and $\tSigma$ be any lifts of $\Gamma$ and $\Sigma$ to $\RRR$.

We fix a leaf $\tW$ of the foliation $\tcF^{cu}$. By Theorem \ref{potries theorem}, given any $\tx \in \RRR$, $\tW^s(\tx)$ intersects $W$ in precisely one point, and this point varies continuously with $\tx$.  Denote by $P: \RRR \to \R$ the projection to the last coordinate of a point in $\RRR$ expressed in the basis $\{[\gamma], [\sigma], v \}$. That is, $P(\tx) = x_3$ where $\tx =  x_1 [\gamma]+x_2 [\sigma]+x_3 v$. We know from Section \ref{preliminaries} that $P\circ H$ maps $\tW^s(\tx)$ homeomorphically onto $\R$ for every $\tx \in \RRR$. In particular, the map
\begin{align}
\Psi: \RRR & \to \tW \times \R \\
        \tx & \mapsto (\tW^s(\tx) \cap \tW , P\circ H(\tx))
\end{align}
is a continuous bijection.

\begin{lemma}
The map $\Psi$ is a homeomorphism.
\end{lemma}

\begin{proof}
We only have to show that $\Psi$ is proper, i.e., for any compact set $K\subset \tW \times \R$, $\Psi^{-1}(K)$ is bounded. Since $K$ is compact, there exists a compact set $B\subset \tW$ and compact interval $J\subset\R$ such that $K\subset B\times J$. $B$ is also compact as subset of $\RRR$. Define a map $\tau: B\to\R$ by projecting $\Psi(p)$ to the second coordinate for any $p\in B$. Obviously, $\tau$ is continuous. Hence $\tau(B)$ is bounded. Given $(p, s), (q,t)\in K$, let $ \tx=\Psi^{-1}(p,s)$, $\ty=\Psi^{-1}(q,t)$. Then
$$
\|\tx-\ty\|\le \|\tx-p\|+\|p-q\|+\|q-\ty\|.
$$
We will estimate the three terms in the above formula. Assume that for any $\tz\in\RRR, \|H\tz-\tz\|\le K_1$. Since $\tx$ and $p$ are in the same strong stable leaf of $\tf$, $H\tx$ and $Hp$ are in the same strong stable leaf of $A$, i.e., if $H\tx=x_1[\gamma]+x_2[\sigma]+s v$, then $Hp=x_1[\gamma]+x_2[\sigma]+s'v$. This implies $\|H\tx-Hp\|=|PH\tx-PHp|$ (noting that we assume that $v$ is a unit vector). Thus,
\begin{eqnarray*}
\|\tx-p\|&\le&\|\tx-H\tx\|+\|H\tx-Hp\|+\|Hp-p\|\\
&\le& 2K_1+|PH\tx-PHp|\\
&\le& 2K_1+|PH\tx|+|PHp|\\
&=& 2K_1+|PH\tx|+|\tau p|.
\end{eqnarray*}
Since $PH\tx\in J$ and $\tau(B)$ is bounded, $\|\tx-p\|$ is bounded for $(p,s)\in K$. Similarly, we have $\|q-\ty\|$ is bounded. Since $B$ is compact, $\|p-q\|$ is obviously bounded. This proves that $\|\tx-\ty\|$ is bounded and hence $\Psi^{-1}(K)$ is bounded.
\end{proof}

Let $T_t: \tW \times \R \to \tW \times \R$ be the map $(p,s) \mapsto (p,s+t)$. We define
\begin{align}
\psi: \RR & \to \RRR \\
(s,t) & \mapsto \Psi^{-1} T_t \Psi \tSigma(s).
\end{align}

\begin{lemma}
There exists $K>0$ such that $\| s [\sigma]+ t v - \psi(s,t) \| <K$ for every $(s,t) \in \RR$.
\end{lemma}

\begin{proof}
Since $H$ is of bounded $C^0$ distance from the identity, there exists $K_1>0$ such that for any $\tx\in\RRR$, $\|H(\tx)-\tx\|\le K_1$. And for any $s\in\R$, we may assume that $\|\tSigma(s)-s[\sigma]\|\le K_1$.
Let $\Psi\tSigma(s)=(p, \tau)$. Then $T_t\Psi\tSigma(s)=(p,\tau+t)$. Denote $\tx=\tSigma(s)=\psi(s,0)=\Psi^{-1}(p,\tau)$ and $\ty=\psi(s,t)=\Psi^{-1} T_t \Psi \tSigma(s)=\Psi^{-1}(p,\tau+t)$. Since $\tx$ and $\ty$ are in the same strong stable leaf of $\tf$, $H\tx$ and $H\ty$ are in the same strong stable leaf of $A$, i.e., if $H\tx=x_1[\gamma]+x_2[\sigma]+\tau v$, then $H\ty=x_1[\gamma]+x_2[\sigma]+(\tau+t)v$. This implies that $H\ty-H\tx=t v$, i.e., $H\psi(s,t)-H\psi(s,0)=t v$. Then
\begin{eqnarray*}
\|\psi(s,t)-s[\sigma]-tv\|&\le& \|\psi(s,t)-H\psi(s,t)\|+\\\|H\psi(s,t)-H\psi(s,0)-tv\| &+&\|H\psi(s,0)-\psi(s,0)\|+\|\psi(s,0)-s[\sigma]\|\\
&\le& K_1+0+K_1+K_1=3K_1.
\end{eqnarray*}
Taking $K=3K_1$ finishes the proof of the lemma.
\end{proof}
Applying Theorem \ref{intersection} with $X = [\gamma]$, $Y = [\sigma]$, $Z = v$ and $\xi = \tGamma$, we conclude that there exists $r,s,t \in \R$ such that $\tGamma(r) = \psi(s,t)$. Notice that $\psi(s,\cdot)$ maps $\R$ homeomorphically onto $\tW^s(\tSigma(s))$. That means that $\tGamma(r)$ and $\tSigma(s)$ belong to the same strong stable leaf of $\tf$. Therefore, $\Gamma(r)$ and $\Sigma(s)$ belong to the same strong stable leaf of $f$. This is a contradiction, since  $U$ and $V$ are disjoint, saturated by stable leaves,   $\Gamma$ lies in $U$, and  $\Sigma$ lies in $V$. We conclude that $f$ is transitive.

\bibliographystyle{plain}

\bibliography{transitive}

\end{document}